\documentclass[12pt]{article}
\usepackage{amsmath,amssymb,amsthm,eucal, mathtools, mathrsfs}
\usepackage[T1]{fontenc}
\usepackage{color}
\usepackage[all]{xy}
\usepackage[pdftex, draft=false]{hyperref} 
\title{\vspace*{-1pc}%
The Friedrichs angle and alternating projections in Hilbert $C^{*}$-modules}
\author{Bram Mesland\S$^*$, Adam Rennie\dag
\thanks{email: 
\texttt{b.mesland@math.leidenuniv.nl}, \texttt{renniea@uow.edu.au}
}
\\[3pt]
\S Mathematisch Instituut, Universiteit Leiden, Netherlands
\\[3pt]
\dag School of Mathematics and Applied Statistics, University of Wollongong\\
Wollongong, Australia
}

\advance\topmargin by -\headheight
\advance\topmargin by -\headsep
\evensidemargin=\oddsidemargin
\makeatletter
\def\section{\@startsection{section}{1}{\z@}{-3.5ex plus -1ex minus
  -.2ex}{2.3ex plus .2ex}{\large\bf}}
\def\subsection{\@startsection{subsection}{2}{\z@}{-3.25ex plus -1ex
  minus -.2ex}{1.5ex plus .2ex}{\normalsize\bf}}
\makeatother

\numberwithin{equation}{section} 

\theoremstyle{plain} 
\newtheorem{thm}{Theorem}[section]
\newtheorem{lemma}[thm]{Lemma}
\newtheorem{prop}[thm]{Proposition}
\newtheorem{corl}[thm]{Corollary}

\theoremstyle{definition} 
\newtheorem{defn}[thm]{Definition}
\newtheorem*{ass*}{Standing assumption}
 
\theoremstyle{remark} 
\newtheorem{rmk}[thm]{Remark}

\DeclareMathOperator{\End}{End}   

\newcommand{\C}{\mathbb{C}}   

\newcommand{\Ran}{\textnormal{Ran }}

\newcommand{\stroke}{\mathbin|}     

\def\pairL_#1(#2|#3){{}_{#1}(#2\stroke#3)} 
\def\pairR(#1|#2)_#3{(#1\stroke#2)_{#3}} 
\def\scal<#1|#2>{\langle#1\stroke#2\rangle} 

\newbox\ncintdbox \newbox\ncinttbox 
	\setbox0=\hbox{$-$}
	\setbox2=\hbox{$\displaystyle\int$}
	\setbox\ncintdbox=\hbox{\rlap{\hbox
		to \wd2{\hskip-.125em \box2\relax\hfil}}\box0\kern.1em}
	\setbox0=\hbox{$\vcenter{\hrule width 4pt}$}
	\setbox2=\hbox{$\textstyle\int$}
	\setbox\ncinttbox=\hbox{\rlap{\hbox
		to \wd2{\hskip-.175em \box2\relax\hfil}}\box0\kern.1em}

\begin{document}

\maketitle

\vspace{-2pc}

\begin{abstract}
Let $B$ be a $C^{*}$-algebra, $X$ a Hilbert $C^{*}$-module over $B$ and $M,N\subset X$ a pair of complemented submodules. We prove the $C^{*}$-module version of von Neumann's alternating projections theorem: the sequence $(P_{N}P_{M})^{n}$ is Cauchy in the $*$-strong module topology if and only if $M\cap N$ is the complement of $\overline{M^{\perp}+N^{\perp}}$. In this case, the $*$-strong limit of $(P_{M}P_{N})^{n}$ is the orthogonal projection onto $M\cap N$. We use this result and the local-global principle to show that the cosine of the Friedrichs angle $c(M,N)$ between any pair of complemented submodules $M,N\subset X$ is well-defined and that $c(M,N)<1$ if and only if $M\cap N$ is complemented and $M+N$ is closed. 
\end{abstract}



\parskip=6pt
\parindent=0pt
\allowdisplaybreaks
\section*{Introduction}

In this note we offer a new and general approach to the two projection problem in Hilbert $C^{*}$-modules.
As an application we extend and improve upon several of the main results in the recent work of \cite{Luo} by giving new proofs that allow for the removal of a key hypothesis. 

Briefly, we begin by proving the Hilbert $C^{*}$-module version of von Neumann's alternating projections theorem, which computes the projection onto $M\cap N$ for a concordant pair of complemented submodules $M,N$ (see below). We then proceed to use this result to define the Friedrichs angle between an arbitrary pair of complemented submodules. The angle is realised as a function on the space of representations of the coefficient algebra of the module. The properties of the Friedrichs angle give necessary and sufficient conditions for the sum and intersection of two complemented submodules to again be complemented.
We now give a little more detail on these results. 

Given two closed subspaces $M,N$ of a Hilbert space $H$ there is an orthogonal direct sum decomposition
\begin{equation}
\label{directsum}
H=(M\cap N)\oplus \overline{(M^{\perp}+N^{\perp})}.
\end{equation}
A fundamental result of von Neumann, the \emph{method of alternating projections}, states that the projection $P_{M\cap N}$ onto $M\cap N$ can be obtained as the $*$-strong limit
\[P_{M\cap N}=s-\lim_{n\to\infty} (P_{M}P_{N})^{n}=s-\lim_{n\to\infty} (P_{N}P_{M})^{n}.\]
The (cosine) of the \emph{Friedrichs angle between $M$ and $N$} is the quantity 
$$
c(M,N):=\|P_{M}P_{N}-P_{M\cap N}\|,
$$ 
and the subspace $M^{\perp}+N^{\perp}$ is closed if and only if $c(M,N)<1$.

In this paper we consider a pair $(M,N)$ of complemented submodules of a Hilbert $C^{*}$-module $X$ over a $C^{*}$-algebra $B$. It is well-known that closed submodules of Hilbert $C^*$-modules need not be orthogonally complemented. This one technical constraint necessitates the discussion of adjointable endomorphisms and regular (unbounded) operators for these modules, \cite{FL,Lance}. 

The complementability issue does not arise for finite dimensional vector spaces of course, but already in the case of finite rank, locally trivial vector bundles on compact Hausdorff base spaces we see examples 
of non-complementability of intersections. Classically the issue gives rise to the notion of a strict homomorphism of vector bundles \cite[Section 1.3]{Atiyah}, and we relate the vector bundle situation to the complementability problem in Remarks \ref{eg:VB-unstrict}, \ref{discontangle} and \ref{Atiyah} below.

In Theorem \ref{vN} we show that the pair $(M,N)$ induces a direct sum decomposition like \eqref{directsum} of the Hilbert $C^*$-module $X$ if and only if von Neumann's theorem on alternating projections is valid for this pair of submodules. We call such pairs \emph{concordant} and characterise them in terms of their Hilbert space localisations in Theorem \ref{locharm}. Our results have implications for the Hilbert module version of the two projection problem, \cite{Luo}. The Hilbert space version first gained prominence in the work of Halmos \cite{H}, and has since had numerous incarnations and applications: for a recent survey see \cite{BS}.

In \cite{Luo}, the Friedrichs angle between complemented submodules has been defined under the constraint that $M\cap N$ is complemented. In Section \ref{sec:angle} of this note we remove this hypothesis and extend the definition of the Friedrichs angle to arbitrary pairs of complemented submodules via the local-global principle of \cite{Pierrot}. We interpret the Friedrichs angle as a function on the space $\widehat{B}$ of irreducible representations of $B$ and prove that $c(M,N)=c(M^{\perp},N^{\perp})$. We deduce that $c(M,N)<1$ if and only if the sequence $(P_{N}P_{M})^{n}$ is Cauchy for the operator norm if and only if $M\cap N$ is complemented and $M^{\perp}+N^{\perp}$ is closed.

{\bf Notation.} For a Hilbert $C^{*}$-module $X$ over a $C^{*}$-algebra $B$ we denote by $\End^{*}_{B}(X)$ the unital $C^{*}$-algebra of adjointable operators on $X$ and by $\mathbb{K}(X)\subset \End^{*}_{B}(X)$ the ideal of compact operators. The symbols $\otimes^{\textnormal{alg}}_{B}, \widehat{\otimes}_{B}$ and $\otimes_{B}$ denote the balanced algebraic, projective and $C^{*}$-module tensor products, respectively.

{\bf Acknowledgements.} We thank Marcel de Jeu for helpful conversations and Michael Frank for valuable correspondence.
\section{Concordant submodules}
Let $X$ be a Hilbert $C^{*}$-module over the $C^{*}$-algebra $B$.
Given two complemented submodules $M$ and $N$ of $X$, we write $P_{M},P_{N}$ respectively for the corresponding projections in $\End^{*}_{B}(X)$. The intersection $M\cap N$ is a closed submodule of $X$, and there is an inclusion
\[M^{\perp}+N^{\perp}\subset (M\cap N)^{\perp}.\]
The submodule $M^{\perp}+N^{\perp}$ need not be closed, but since $(M\cap N)^{\perp}$ is closed,
\[\overline{M^{\perp}+N^{\perp}}\subset (M\cap N)^{\perp},\]
as well. In case $X$ is a Hilbert space there is an equality (see (\cite[Theorem 4.6.4]{Dbook})
\begin{equation}
\label{concordant}
 \overline{M^{\perp}+N^{\perp}}= (M\cap N)^{\perp},
\end{equation} 
and thus the projections $P_{M\cap N}$ and $P_{\overline{M^{\perp}+N^{\perp}}}$ exist and satisfy $1-P_{M\cap N}=P_{\overline{M^{\perp}+N^{\perp}}}$. 

In general, the projections do not exist unless the submodules are complemented.  To our knowledge, it is an open question whether the intersection of complemented submodules is again complemented. In \cite[Section 3]{Luo} it was shown that even in case all the projections exist, \eqref{concordant} need not not hold (see Remark \ref{LMXcounter} below).
\begin{defn}
Let $M$ and $N$ be complemented submodules of a Hilbert $C^{*}$-module $X$. The pair $(M,N)$ is \emph{concordant} if $X=(M\cap N)\oplus\overline{(M^{\perp}+N^{\perp})}$. If the pair $(M,N)$ is not concordant, we say it is \emph{discordant}. 
\end{defn}
The pair $(M,N)$ is concordant if their intersection $M\cap N$ is complemented and its complement is $\overline{M^{\perp}+N^{\perp}}$.
\begin{rmk}\label{LMXcounter} The pair $(M,N)$ being concordant is strictly stronger than the requirement that $M\cap N$ be complemented. In \cite[Section 3]{Luo} it is shown that for $X=B=C([0,\frac{\pi}{2}], M_{2}(\mathbb{C}))$, the submodules
\[M=\Ran \begin{pmatrix} 1 & 0 \\ 0 & 0 \end{pmatrix},\quad N= \Ran \begin{pmatrix} \cos^{2}t & \sin t \cos t \\ \sin t \cos t &\sin^{2}t\end{pmatrix},\]
satisfy $M\cap N =0$, which is complemented, whereas $\overline{M^{\perp}+N^{\perp}}\neq X$ so $(M,N)$ is not concordant.
\end{rmk}

\begin{rmk} 
Note that $(M,N)$ is \emph{harmonious} in the sense of \cite[Definition 4.1]{Luo} if each of the submodules
\[
\overline{M+N},\,\, \overline{M+N^{\perp}},\,\,\overline{M^{\perp}+N},\,\, \overline{M^{\perp}+N^{\perp}} 
\]
is complemented. In this case the respective complements are
\[
M^{\perp}\cap N^{\perp},\,\,M^{\perp}\cap N,\,\, M+N^{\perp},\,\, M\cap N,
\]
as explained in the discussion after \cite[Definition 4.1]{Luo}. Thus $(M,N)$ is harmonious if and only if each of the pairs $(M,N)$, $(M,N^{\perp})$, $(M^{\perp},N)$ and $(M^{\perp},N^{\perp})$ is concordant.
 \end{rmk}
 \begin{rmk} If $M+N$ is closed, then by \cite[Proposition 4.6]{LSX} $M^{\perp}+N^{\perp}$ is closed and $X=(M\cap N)\oplus (M^{\perp}+N^{\perp})$. In particular, $M+N$ is closed if and only if $M^{\perp}+N^{\perp}$ is closed and in this case both $(M,N)$ and $(M^{\perp},N^{\perp})$ are concordant (see Proposition 3.10 below).
 \end{rmk}
 \begin{rmk}
 \label{rmkuniversal} 
 In \cite{RS} it was shown that the universal $C^{*}$-algebra $C^{*}(p,q)$ generated by two projections $p$ and $q$ admits the following concrete model
\[
C^{*}(p,q)\simeq\left\{A(t)\in C([0,\pi/2], M_2(\mathbb{C})): A(0)\ \mbox{and}\ A(\pi/2)\ \mbox{diagonal}\right\},
\]
with the isomorphism is determined by
\[
p\mapsto P:=\begin{pmatrix} 1 & 0 \\ 0 & 0 \end{pmatrix}, \quad q\mapsto Q
:=\begin{pmatrix} \cos^{2}t & \sin t \cos t \\ \sin t \cos t &\sin^{2}t\end{pmatrix}.
\]
From this point of view, the counterexample of \cite[Section 3]{Luo} discussed in Remark \ref{LMXcounter} above arises from the universal example. This shows that specific properties such as being concordant or harmonious hold in some representations of $C^{*}(p,q)$, but not in all of them.
\end{rmk}

We will now characterise concordant pairs by looking at their Hilbert space localisations.

Let $\pi:B\to \mathbb{B}(H_{\pi})$ be a representation of $B$ on the Hilbert space $H_{\pi}$ and write $X_{\pi}:=X\otimes_{B}H_{\pi}$. There is a representation
\begin{equation}
\widehat{\pi}:\End^{*}_{B}(X)\to \mathbb{B}(X_{\pi}),\quad T\mapsto T\otimes 1.
\label{eq:bob}
\end{equation}
Write $M_{\pi}:=M\otimes_{B}H_{\pi}\subset X\otimes_{B}H_{\pi}$, and similarly for $N$. Then $M_{\pi}$ and $N_{\pi}$ are closed subspaces of the Hilbert space $X_{\pi}$ and we have $P_{M_{\pi}}:=\widehat{\pi}(P_{M})=P_{M}\otimes 1$, as well as  $P_{N_{\pi}}:=\widehat{\pi}(P_{N})=P_{N}\otimes 1$. Since the subspace $M_{\pi}\cap N_{\pi}$ is closed, there is a projection $P_{M_{\pi}\cap N_{\pi}}\in\mathbb{B}(X_{\pi})$ that projects onto $M_{\pi}\cap N_{\pi}$. In general, the equality $M_{\pi}\cap N_{\pi}=(M\cap N)_{\pi}$ need not hold, even if $M\cap N$ is complemented. 
We recall the following fact.
\begin{prop}[Local-global principle for complemented submodules \cite{Pierrot}]
\label{locglob} 
Let $\Omega\subset X$ be a closed submodule. Then $\Omega$ is complemented if and only if for every irreducible representation $\pi:B\to \mathbb{B}(H_{\pi})$ there is an equality $(\Omega_{\pi})^{\perp}=(\Omega^{\perp})_{\pi}$.
\end{prop}
\begin{proof}
By \cite[Corollaire 1.17]{Pierrot}, we have that $X=\Omega\oplus \Omega^{\perp}$ if and only if for every irreducible representation $\pi:B\to \mathbb{B}(H_{\pi})$ there is an equality
\[
X_{\pi}=X\otimes_{B}H_{\pi}=(\Omega\oplus \Omega^{\perp})\otimes_{B}H_{\pi}=\Omega\otimes_{B}H_{\pi}\oplus \Omega^{\perp}\otimes_{B}H_{\pi}=\Omega_{\pi}\oplus (\Omega^{\perp})_{\pi}.
\]
Since  $(\Omega^{\perp})_{\pi}\subset (\Omega_{\pi})^{\perp}$, this holds if and only if $(\Omega^{\perp})_{\pi}=(\Omega_{\pi})^{\perp}$.
\end{proof}
A weaker form of this result was proved independently, though several years later, in \cite{KL12}. There, the local side of the equivalence involved \emph{all} representations of the $C^{*}$-algebra $B$. The two results are equivalent because the proof of the implication $\Rightarrow$ in Proposition \ref{locglob} holds verbatim for an arbitrary representation of the $C^{*}$-algebra $B$, see \cite{KL17}. We will use both instances of the result.
\begin{lemma}
\label{inclusion} 
Let $X$ be a Hilbert $C^{*}$-module over $B$, $M,N$ complemented submodules and $\pi:B\to\mathbb{B}(H_{\pi})$ a representation. Then there is an equality of closed subspaces
\[
\overline{(M_{\pi})^{\perp}+(N_{\pi})^{\perp}}=\big(\,\overline{M^{\perp}+N^{\perp}}\,\big)_{\pi}.
\]
\end{lemma}
\begin{proof}
The inclusion of subspaces
\begin{align*}
(M^{\perp})_{\pi}+(N^{\perp})_{\pi}\subset \Big(\overline{M^{\perp}+N^{\perp}}\Big)_{\pi}
\end{align*}
shows that we have an inclusion of closed linear subspaces
\begin{equation*}
\overline{(M^{\perp})_{\pi}+(N^{\perp})_{\pi}}\subset \left(\overline{M^{\perp}+N^{\perp}}\right)_{\pi}.\end{equation*} 
 The subspace $(M^{\perp}+N^{\perp})\otimes^{\textnormal{alg}}_{B}H_{\pi}$ is dense in $(\,\overline{M^{\perp}+N^{\perp}}\,)_{\pi}$ and since
 \[
 (M^{\perp}+N^{\perp})\otimes^{\textnormal{alg}}_{B}H_{\pi}\subset(M^{\perp})_{\pi}+(N^{\perp})_{\pi}\subset \overline{(M^{\perp})_{\pi}+(N^{\perp})_{\pi}}\subset \left(\overline{M^{\perp}+N^{\perp}}\right)_{\pi},\]
 it follows that $\overline{(M^{\perp})_{\pi}+(N^{\perp})_{\pi}}=\left(\overline{M^{\perp}+N^{\perp}}\right)_{\pi}$. Since $M$ and $N$ are complemented we have $(M_{\pi})^{\perp}=(M^{\perp})_{\pi}$ and $(N_{\pi})^{\perp}=(N^{\perp})_{\pi}$ and thus $\overline{(M_{\pi})^{\perp}+(N_{\pi})^{\perp}}=\left(\overline{M^{\perp}+N^{\perp}}\right)_{\pi}$.
\end{proof}
\begin{thm}
\label{locharm}
Let $X$ be a Hilbert $C^{*}$-module over $B$ and $M$ and $N$ complemented submodules. Then 
the pair $(M,N)$ is concordant if and only if for every irreducible 
representation $\pi:B\to \mathbb{B}(H_{\pi})$ there is an equality of closed subspaces $M_{\pi}\cap N_{\pi}=(M\cap N)_{\pi}$.
\end{thm}
\begin{proof} 
Suppose that $M$ and $N$ are concordant so that
\[
X=(M\cap N)\oplus \left(\overline{M^{\perp}+N^{\perp}}\right).
\]
Therefore Proposition \ref{locglob} and Lemma \ref{inclusion} give
 \[
 ((M\cap N)_{\pi})^{\perp}=((M\cap N)^{\perp})_{\pi}=(\overline{M^{\perp}+N^{\perp}})_{\pi}=\overline{(M_{\pi})^{\perp}+(N_{\pi})^{\perp}}.
 \]  

Taking orthogonal complements we find 
$(M\cap N)_{\pi}=\Big(\overline{(M_{\pi})^{\perp}+(N_{\pi})^{\perp}}\Big)^{\perp}=M_{\pi}\cap N_{\pi}$.

Conversely, suppose that $M_{\pi}\cap N_{\pi}=(M\cap N)_{\pi}$ for all irreducible 
representations $\pi$. 
By Lemma \ref{inclusion} and Equation \eqref{concordant} we have 
\[
(\overline{M^{\perp}+N^{\perp}})_{\pi}= \overline{(M_{\pi})^{\perp}+(N_{\pi})^{\perp}}=(M_{\pi}\cap N_{\pi})^{\perp},
\] 

and we deduce that
\begin{align*}
(M\cap N)_{\pi}\oplus (\overline{M^{\perp}+N^{\perp}})_{\pi}&= (M_{\pi}\cap N_{\pi})\oplus (M_{\pi}\cap N_{\pi})^{\perp}=X_{\pi}.
\end{align*}

By Proposition \ref{locglob} we conclude that $X=(M\cap N)\oplus \overline{M^{\perp}+N^{\perp}}.$
\end{proof}
In line with the local-global principle, Proposition \ref{locglob}, we obtain the same result when we consider all representations of the base algebra $B$.
\begin{corl}
\label{locharmcor}
Let $X$ Hilbert $C^{*}$-module over $B$ and $M$ and $N$ complemented submodules. Then $(M,N)$ is concordant if and only if for every representation $\pi:B\to \mathbb{B}(H_{\pi})$ there is an equality of closed subspaces $M_{\pi}\cap N_{\pi}=(M\cap N)_{\pi}$.
\end{corl}
\begin{proof} 
The proof of $\Rightarrow$ in Theorem \ref{locharm} shows that $M_{\pi}\cap N_{\pi}=(M\cap N)_{\pi}$ for every representation whenever $(M,N)$ is concordant.
\end{proof}
\begin{rmk}
\label{eg:VB-unstrict}
Consider $B=C([0,\frac{\pi}{2}])$, $X=C([0,\frac{\pi}{2}], \mathbb{C}^{2})$ and consider the submodules
\[
M=\Ran\!\!\begin{pmatrix} 1 & 0 \\ 0 & 0 \end{pmatrix},\quad N
= \Ran\!\!\begin{pmatrix} \cos^{2}t & \sin t \cos t \\ \sin t \cos t &\sin^{2}t\end{pmatrix}.
\] We have $M\cap N=0$ and for the irreducible representations given by $t\in[0,\pi/2]$ we have
\[
M_t\cap N_t=\begin{cases} 0 &  t=0\\ \C & t\neq 0,\end{cases}
\]
so $(M,N)$ is discordant by Theorem \ref{locharm}.
\end{rmk}


\section{Von Neumann's theorem of alternating projections}
Let $P,Q\in \End^{*}_{B}(X)$ be projections
\[
P^{*}=P^{2}=P,\quad Q^{*}=Q^{2}=Q.
\]
The submodules $\Ran P$ and $\Ran Q$ are complemented in $X$, and every complemented submodule is the range of an adjointable projection. As noted before, it is an open question whether the intersection $\Omega:=\Ran P\cap \Ran Q$, which is a closed submodule, is complemented. In case $B=\mathbb{C}$ and $X$ is a Hilbert space this is true and thus there is a projection $P_{\Omega}$ with $\Ran P_{\Omega}=\Omega$. For $n\geq 0$, write
\[
(P,Q)_{n}:=\cdots PQPQ,\quad \textnormal{the product of exactly } n \,\,\textnormal{alternating factors ending in } Q.
\]
Von Neumann proved the following well-known theorem.
\begin{thm}[{\cite[Lemma 22]{vN}}]
\label{vN} 
Let $H$ be a Hilbert space, $M,N\subset H$ closed subspaces and $\Omega:=M\cap N$. Let $P=P_{M}$ and $Q=P_{N}$ be the orthogonal projections onto $M$ and $N$ respectively. The orthogonal projection $P_{\Omega}$ onto $\Omega$ can be obtained as the strong limit of any of the sequences
\begin{equation}
\label{alternatingseq}
(PQ)^{n},\quad (QP)^{n}, \quad (P,Q)_{n},\quad (Q,P)_{n},
\end{equation}
or any of their subsequences. 
\end{thm}
In a Hilbert $C^{*}$-module $X$, the analogue of the $*$-strong topology is  defined by the seminorms
\[
\|T\|_{x}:=\max\{\|Tx\|,\|T^{*}x\|\},\quad x\in X,
\]
and we refer to this topology as the $*$-\emph{strong module topology}.
On bounded sets the $*$-strong module topology coincides with the 
\emph{strict topology} on $\End_{B}^{*}(X)$ relative to the ideal $\mathbb{K}(X)$, \cite[Proposition 5.5.9]{Troitsky}.
The following fact is well-known.
\begin{lemma}\label{complete}
The $*$-strong module topology is complete on bounded sets.
\end{lemma}
\begin{proof} Let $T_{n}\in\End^{*}_{B}(X)$ be a sequence that is Cauchy for the seminorms $\|\cdot \|_{x}$, $x\in X$. By the Uniform Boundedness Principle, the operators \[Tx:=\lim_{n\to\infty} T_{n}x, \quad \textnormal{and} \quad T^{*}x:=\lim_{n\to\infty} T^{*}_{n}x,\] are well-defined, bounded and mutually adjoint.
\end{proof}
\begin{lemma}\label{reductionlemma} Let $P,Q\in\End^{*}_{B}(X)$ be projections. Then $(PQ)^{n}$ and $(QP)^{n}$ are $*$-strongly Cauchy if and only if $(PQP)^{n}$ and $(QPQ)^{n}$ are $*$-strongly Cauchy  if and only if $(P,Q)_{n}$ and $(Q,P)_{n}$ (as defined in \eqref{alternatingseq}) are $*$-strongly Cauchy. The same statement holds for the norm topology.
\end{lemma}
\begin{proof} 
Since
\[
(P,Q)_{n}=\left\{\begin{matrix}  (PQ)^{\frac{n}{2}} & n\quad \textnormal{even}\\ (QPQ)^{\frac{n-1}{2}} & n\quad  \textnormal{odd},\end{matrix}\right.\quad (Q,P)_{n}=\left\{\begin{matrix}  (QP)^{\frac{n}{2}} & n\quad \textnormal{even}\\ (PQP)^{\frac{n-1}{2}} & n\quad  \textnormal{odd},\end{matrix}\right. 
\]
it suffices to prove that $(PQ)^{n}$ and $(QP)^{n}$ are $*$-strongly Cauchy if and only $(PQP)^{n}$ and $(QPQ)^{n}$ are $*$-strongly Cauchy. The same holds for the norm topology.

Any projection $P$ satisfies $\langle Px,Px\rangle\leq \langle x,x\rangle$ and $Q(PQ)^{n}=(QPQ)^{n}$ so that
\begin{align}
\label{sandwich}
\langle (PQ)^{n}x,(PQ)^{n}x\rangle&=\langle (Q(PQ)^{n}+(1-Q)(PQ)^{n})x,(Q(PQ)^{n}+(1-Q)(PQ)^{n})x\rangle\nonumber\\
&\geq \langle (QPQ)^{n}x,(QPQ)^{n}x\rangle\nonumber\\
&=\langle (P(QPQ)^{n}+(1-P)(QPQ)^{n})x,(P(QPQ)^{n}+(1-P)(QPQ)^{n})x\rangle\nonumber\\
&\geq \langle (PQ)^{n+1}x,(PQ)^{n+1}x\rangle.
\end{align}
Now for $m>n$ we have 
\[
(PQ)^{n}-(PQ)^{m}=(PQ)^{n}(1-(PQ)^{m-n}),\quad  ((QPQ)^{n}-(QPQ)^{m})=(QPQ)^{n}(1-(QPQ)^{m-n}),
\]
which, together with \eqref{sandwich} gives
\begin{align*}
\langle ((PQ)^{n}-(PQ)^{m})x ,((PQ)^{n}&-(PQ)^{m})x\rangle =
\langle (PQ)^{n}(1-(PQ)^{m-n})x,(PQ)^{n}(1-(PQ)^{m-n})x\rangle \\ &\geq \langle (QPQ)^{n}(1-(PQ)^{m-n})x,(QPQ)^{n}(1-(PQ)^{m-n})x\rangle\\
&= \langle ((QPQ)^{n}-(QPQ)^{m})x,((QPQ)^{n}-(QPQ)^{m})x\rangle\\
&= \langle (QPQ)^{n}(1-(QPQ)^{m-n})x,(QPQ)^{n}(1-(QPQ)^{m-n})x\rangle\\
&\geq \langle (PQ)^{n+1}(1-(QPQ)^{m-n})x,(PQ)^{n+1}(1-(QPQ)^{m-n})x\rangle\\
&\geq \langle( (PQ)^{n+1}-(PQ)^{m+1})x,((PQ)^{n+1}-(PQ)^{m+1})x\rangle.
\end{align*}
This proves that $(PQ)^{n}$ is pointwise Cauchy if and only if $(QPQ)^{n}$ is pointwise Cauchy. Thus $(PQ)^{n}$ and $(QP)^{n}$ are both $*$-strongly Cauchy if and only if $(PQP)^{n}$ and $(QPQ)^{n}$ are both $*$-strongly Cauchy. The statements for the norm topology follow from the same inequalities. This completes the proof.
\end{proof}
\begin{prop} 
\label{Cauchy} 
Suppose that $(PQ)^{n}$ is $*$-strongly Cauchy. Then so are $(QP)^{n}$, $(PQP)^{n}$, $(QPQ)^{n}$, $(Q,P)_{n}$ and $(P,Q)_{n}$. The $*$-strong limits of each of these sequences is a projection $P_{\Omega}$ with range $\Omega:=\Ran P \cap \Ran Q$. In particular $\Omega$ is complemented.
\end{prop}
\begin{proof} 
Since $((PQ)^{n})^{*}=(QP)^{n}$, the first statement follows from Lemma \ref{reductionlemma}. 
We will prove that $s-\lim_{n\to\infty} (PQP)^n=s-\lim_{n\to \infty} (QPQ)^{n}$ and that this operator is a projection $P_{\Omega}$ with range $\Omega$. It then follows that $\Omega$ is complemented and that 
\[
P_{\Omega}=s-\lim_{n\to\infty}(P,Q)_{n}=s-\lim_{n\to\infty}(Q,P)_{n},
\]
since $(PQP)^{n}$ is a subsequence of $(Q,P)_{n}$ and $(QPQ)^{n}$ is a subsequence of $(P,Q)_{n}$. Then $(PQ)^{n}$ and $(QP)^{n}$ are subsequences of $(P,Q)_{n}$ and $(Q,P)_{n}$, respectively it follows that
\[
P_{\Omega}=s-\lim_{n\to\infty}(PQ)^{n}=s-\lim_{n\to\infty}(QP)^{n},
\]
as well.

By Lemma \ref{complete} the $*$-strong limit $\tilde{P}:=\lim (PQP)^{n}$ exists, is self-adjoint and $\|\tilde{P}\|\leq 1$. To prove that $\tilde{P}$ is a projection let $x\in X$ and $\varepsilon >0$. Choose $N$ such that for all $k\geq N$ we have
\[
\|\tilde{P}x-(PQP)^{k}x\|<\varepsilon.
\]
Now consider
\begin{align*}
\|\tilde{P}^{2}x-\tilde{P}x\|
&=\|\tilde{P}(PQP)^{k}x-\tilde{P}x\|+ \|\tilde{P}(\tilde{P}-(PQP)^{k})x\|\\
&\leq \|\tilde{P}(PQP)^{k}x-\tilde{P}x\|+ \|(\tilde{P}-(PQP)^{k})x\|\\
&< \|\tilde{P}(PQP)^{k}x-\tilde{P}x\|+ \varepsilon\\
&=\lim_{n\to\infty}\|(PQP)^{n+k}x-\tilde{P}x\|+\varepsilon=\varepsilon,
\end{align*}
and as $\varepsilon$ was arbitrary, it follows that $\tilde{P}^{2}x=\tilde{P}x$.

To prove that $\Ran\tilde{P}=\Omega$, first observe that if $x\in \Omega$ then $$x=Px=Qx=PQPx,$$ so $\tilde{P}x=x$ and $\Omega\subset \Ran \tilde{P}$. 

For the reverse inclusion we will show that $\tilde{P}=P\tilde{P}=Q\tilde{P}$. The equalities 
\[
P\tilde{P}x=\tilde{P}x,\quad \textnormal{and}\quad PQ\tilde{P}x=\tilde{P}x,
\] 
hold by construction. Now for any $x\in X$ we have
\[
\langle Px, Px\rangle\leq \langle x,x\rangle,\quad\langle Qx, Qx\rangle\leq \langle x,x\rangle,
\]
 from which we deduce that
 \[
 \langle\tilde{P}x,\tilde{P}x\rangle=\langle PQ\tilde{P}x,PQ\tilde{P}x\rangle\leq \langle Q\tilde{P}x,Q\tilde{P}x\rangle\leq  \langle\tilde{P}x,\tilde{P}x\rangle.
 \]
Therefore $\langle Q\tilde{P}x,Q\tilde{P}x\rangle= \langle\tilde{P}x,\tilde{P}x\rangle$ and $\langle (1-Q)\tilde{P}x,(1-Q)\tilde{P}x\rangle=0$. It follows that $(1-Q)\tilde{P}x=0$ so $Q\tilde{P}x=\tilde{P}x$. This shows that $Q\tilde{P}=\tilde{P}$ and thus $\Ran \tilde{P}\subset\Omega$. Therefore $\Omega$ is complemented and $P_{\Omega}=\tilde{P}=s-\lim (PQP)^{n}$ in the $*$-strong module topology. By exhanging the r\^oles of $P$ and $Q$, we find that $P_{\Omega}=s-\lim (QPQ)^{n}$ as well.
\end{proof}
In order to address the appropriate converse to Proposition \ref{Cauchy}, we need a description of the Banach space dual $X^{*}:=\mathbb{B}(X,\mathbb{C})$ of bounded linear functionals on a Hilbert $C^{*}$-module $X$. To this end we first recall the \emph{dual or conjugate $C^{*}$-module}. 

The space of compact operators $\mathbb{K}(X,B)$ from $X$ to $B$ is a left $B$-module via $(b\cdot K)(x):=bK(x)$ and carries a natural left $B\simeq \mathbb{K}(B,B)$ valued inner product $\langle K, L\rangle:=KL^{*}$. The \emph{conjugate module} $\overline{X}$ is defined to be the set $X$ with the conjugate $\mathbb{C}$-vector space structure, and we write elements of $\overline{X}$ as $\overline{x}$ with $x\in X$. The left $B$-module structure and inner product
\[b\cdot \overline{x}:=\overline{xb^{*}},\quad \langle \overline{x},\overline{y}\rangle:=\langle x,y\rangle.\]
These left Hilbert $C^{*}$-modules over $B$ are isomorphic, by the following well-known theorem \cite[page 13]{Lance}.
\begin{prop}[Riesz-Fr\'echet theorem for Hilbert $C^{*}$-modules]
The map 
\[
T:\overline{X}\rightarrow \mathbb{K}(X,B),\quad \overline{x}\mapsto T_{x},\quad T_{x}(y):=\langle x,y\rangle, \quad x,y\in X, 
\]
is a unitary isomorphism of left Hilbert $C^{*}$-modules over $B$.
\end{prop}
The dual Banach space of the $C^{*}$-algebra $B$, $B^{*}:=\mathbb{B}(B,\mathbb{C})$, is a right Banach $B$-module via
\[
(\varphi\cdot b) (a):=\varphi(ba),\quad a,b\in B.
\]
Lastly, for a right Banach $B$-module $V$ and a left Banach $B$-module $W$, we denote by $V\widehat{\otimes}_{B}W$ the balanced Banach space projective tensor product of $V$ and $W$. We are now ready to recall a result of Schweizer, \cite[Proposition 3.1]{Schweizer},  giving a complete description of the dual Banach space $X^{*}$ of the module $X$.
\begin{prop}\label{Schweizer} 
Let $X$ be a Hilbert $C^{*}$-module, $\overline{X}:=\mathbb{K}(X,B)$ the conjugate module and $X^{*}=\mathbb{B}(X,\mathbb{C})$ the dual Banach space of $X$. The  map $\psi:B^{*}{\otimes}_{B}^{\textnormal{alg}}\overline{X}\to X^{*}$ given by
\[
\psi(\phi\otimes\overline{y})(x):=\phi(\langle y,x\rangle), \quad \phi\in B^{*},\quad x,y\in X,
\]
extends to an isometric isomorphism $B^{*}\widehat{\otimes}_{B}\overline{X}\to X^{*}$ of Banach spaces. 
\end{prop}
For a Banach space $W$, the \emph{weak topology} on $W$ is the locally convex topology defined by the seminorms $\|w\|_{\varphi}:=\|\varphi(w)\|$. In general the weak topology is \emph{not} complete, that is, weak Cauchy sequences need not have a weak limit in $X$. However, we do have the following fundamental result for weakly convergent sequences. 
\begin{thm}[{\cite[Chap II, Section 38]{RSz}}]
\label{Mazur}
Let $W$ be a Banach space and $C\subset W$ a convex set. Then the weak closure of $C$ coincides with the norm closure of $C$. In particular, if $w_j\to w$ in the weak topology, then there exists a sequence of convex combinations $y_j:=\sum_{k=j}^{n_{j}} t_j w_j$ such that $\|y_j-w\|\to 0$.
\end{thm}
In the sequel we will freely use the following computational tool.
\begin{lemma}
\label{powers} 
Let $P,Q\in\End^{*}_{B}(X)$ be projections such that $\Omega:=\Ran P\cap \Ran Q$ is complemented. Then for all $k\geq 1$ we have
\begin{align*}
(PQ-P_{\Omega})^{k}&=(PQP)^{k}-P_{\Omega},\quad (QP-P_{\Omega})^{k}=(QPQ)^{k}-P_{\Omega},\\
(PQP-P_{\Omega})^{k}&=(PQP)^{k}-P_{\Omega},\quad (QPQ-P_{\Omega})^{k}=(QPQ)^{k}-P_{\Omega}.
\end{align*}
\end{lemma}
\begin{proof}
The statement holds for $k=1$. Since $P_{\Omega}=P_{\Omega}P=PP_{\Omega}=P_{\Omega} Q=QP_{\Omega}$ we have 
\[
(PQ)^{k+1}-P_{\Omega}=(PQ-P_{\Omega})((PQ)^{k}-P_{\Omega}),
\quad (QP)^{k+1}-P_{\Omega}=(QP-P_{\Omega})((QP)^{k}-P_{\Omega}),
\]
and
\[(PQP)^{k+1}-P_{\Omega}=P((QP)^{k+1}-P_{\Omega}),\quad (QPQ)^{k+1}-P_{\Omega}=Q((PQ)^{k+1}-P_{\Omega}),\]
so the result follows by induction on $k$.
\end{proof}
We are now ready to prove our main theorem.

\begin{thm}
\label{complementedalternating} 
Let $M,N$ be complemented submodules of a Hilbert $C^{*}$-module $X$. Then $(M,N)$ is a concordant pair if and only if the sequence $(P_{N}P_{M})^{n}$ is Cauchy in the $*$-strong module topology on $\End^{*}_{B}(X)$.
\end{thm}
\begin{proof} 
We write $P=P_{M}$, $Q=P_{N}$ and $\Omega:=M\cap N$.

$\Leftarrow$ In Proposition \ref{Cauchy} it was proved that $\Omega$ is complemented and $\lim (PQ)^{n}x=P_{\Omega}x$. Now if $\pi:B\to\mathbb{B}(H_{\pi})$ is an irreducible representation then
\begin{align*}
P_{M_{\pi}\cap N_{\pi}}(x\otimes h)&=\lim_{n\to\infty}(P_{M_{\pi}}P_{N_{\pi}})^{n}(x\otimes h)=\lim_{n\to\infty}\widehat{\pi}(P_{M}P_{N})^{n}(x\otimes h)\\
&=\lim_{n\to\infty}((PQ)^{n}x)\otimes h=P_{\Omega}x\otimes h=\widehat{\pi}(P_{\Omega})(x\otimes h),
\end{align*}
so $P_{M_{\pi}\cap N_{\pi}}=\widehat{\pi}(P_{\Omega})$ and thus $\Omega_{\pi}=M_{\pi}\cap N_{\pi}$, so $(M,N)$ is concordant by Theorem \ref{locharm}.

For the converse, assume that $(M,N)$ is concordant and write $P_{\Omega}$ for the projection onto $\Omega$. By Lemma \ref{reductionlemma} it suffices to prove that $(PQP)^{n}x\to P_{\Omega} x$ and $(QPQ)^{n}x\to P_{\Omega} x$ for all $x\in X$. 

We first prove that $(PQP)^{n}x$ converges to $P_{\Omega}x$ in the weak topology on $X$. To this end observe that since $\|(PQP)^{n}\|\leq \|PQP\|^{n}\leq 1$ the sequence $(PQP)^{n}x$ is bounded in norm. Therefore, by Proposition \ref{Schweizer}  it suffices to show that $(\phi\otimes \overline{y}) ( (PQP)^{n}x)\to (\phi\otimes \overline{y})(P_{\Omega}x)$ for all $\phi\in B^{*}$ and $y\in X$, as such functionals generate the weak topology. Since every functional on the $C^{*}$-algebra $B$ is a linear combination of four states (see \cite{T}), we may restrict ourselves to states $\sigma\in B^{*}$. In the universal representation $H_{u}$ of $B$, every state $\sigma$ arises as a vector state associated to a unit vector $h_{\sigma}\in H_{u}$. Denote by $\pi_{\sigma}$ the GNS-representation associated to the state $\sigma$. Then by Theorem \ref{vN} we find
\begin{align*}
(\sigma\otimes \overline{y})((PQP)^{n}x)&=\sigma(\langle y, (PQP^{n})x\rangle )=\langle h_{\sigma}, \langle y, (PQP)^{n}x\rangle h_{\sigma}\rangle\\
&=\langle y\otimes h_{\sigma}, (PQP)^{n}x\otimes h_{\sigma}\rangle \to \langle y\otimes h_{\sigma}, P_{\Omega_{\sigma}} (x\otimes h_{\sigma})\rangle.
\end{align*}
By Corollary \ref{locharmcor}, $M_{\pi_\sigma}\cap N_{\pi_\sigma}=\Omega_{\pi_\sigma}$ so $P_{\Omega_{\sigma}}=P_{\Omega}\otimes 1=\widehat{\pi}_{\sigma}(P_{\Omega})$,
and
$(PQP)^{n}x\otimes h_{\sigma}\to P_{\Omega} x\otimes h_{\sigma}$ in the Hilbert space $X\otimes_{B} H_{u}$. Therefore $(PQP)^{n}x\to P_{\Omega}x$ weakly in $X$. 

By Theorem \ref{Mazur}, there is a sequence of convex combinations $y_{k}=\sum_{i=k}^{n_{k}} t_{i}(PQP)^{i}x$ such that $y_k\to P_{\Omega}x$ in norm in $X$. Since for all $n$ we have
\[
P_{\Omega}(PQP)^{n}=(PQP)^{n}P_{\Omega}=P_{\Omega},\quad (PQP)^{m}\leq (PQP)^{n},\quad m\geq n,
\]
we can estimate
\begin{align*}
\langle (y_{k}-P_{\Omega})x,  (y_{k}-P_{\Omega})x\rangle&=
\left\langle \left(\sum_{i=k}^{n_{k}} t_{i}(PQP)^{i}-P_{\Omega}\right)x , \left(\sum_{i=k}^{n_{k}} t_{i}(PQP)^{i}-P_{\Omega}\right)x\right\rangle\\ 
&=\left\langle \left(\sum_{i=k}^{n_{k}} t_{i}(PQP)^{i}-P_{\Omega}\right)^{2}x , x\right\rangle\\
&=\left\langle \left(\sum_{i,j=k}^{n_{k}}t_{i}t_{j}(PQP)^{i+j}-P_{\Omega}\right)x, x\right\rangle \\ 
&\geq \left\langle \left(\sum_{i,j=k}^{n_k}t_{i}t_{j}(PQP)^{2n_{k}}-P_{\Omega}\right)x, x\right\rangle \\ 
&=\langle ((PQP)^{2n_k}-P_{\Omega})x, x\rangle \\
&=\langle ((PQP)^{n_k}-P_{\Omega})x, ((PQP)^{n_k}-P_{\Omega})x\rangle,
\end{align*}
where the last step follows using Lemma \ref{powers}. 
Therefore it follows that the subsequence $(PQP)^{n_{k}}$ is such that for all $x\in X$  we have norm convergence $(PQP)^{n_{k}}x\to P_{\Omega}x$ as $k\to \infty$. Since for any $m\geq n$ we have 
\[
\langle ((PQP)^{n}-P_{\Omega})x, ((PQP)^{n}-P_{\Omega})x\rangle \geq\langle( (PQP)^{m}-P_{\Omega})x,( (PQP)^{m}-P_{\Omega})x\rangle,
\]
we find that 
\begin{align*}
\|((PQP)^{n}-P_{\Omega})x\|\geq \|((PQP)^{m}-P_{\Omega})x\|.
\end{align*}
Thus it follows that $\lim_{n\to\infty} \|((PQP)^{n}-P_{\Omega})x\|\to 0$. By swapping the r\^oles of $P$ and $Q$ we find that $\lim_{n\to\infty} \|((QPQ)^{n}-P_{\Omega})x\|\to 0$ as well. This completes the proof.
\end{proof}

\section{Angle, sum and intersection}
\label{sec:angle}
We now consider the applications of our main result to various problems concerning pairs of complemented submodules of Hilbert $C^*$-modules.
\subsection{The Friedrichs angle between complemented submodules}
In \cite{Luo}, the following definition for the Friedrichs angle between complemented submodules was given, which we now recall. Let $M,N\subset X$  be complemented submodules such that $M\cap N$ is complemented and write $P_{M},P_{N}$ and $P_{M\cap N}$ respectively for the corresponding projections. The quantity
\begin{equation}
\label{moduleangle}
c(M,N):=\|P_{M}P_{N}(1-P_{M\cap N})\|=\|P_{M}P_{N}-P_{M\cap N}\|,
\end{equation}
is called the (cosine of the) \emph{Friedrichs angle between $M$ and $N$}. 

For the above definition, the existence of the projection $P_{\Omega}$ seems necessary. This is undesirable and ideally the angle should be an invariant associated to any pair $(M,N)$ of complemented submodules. We propose the following generalisation, based on Hilbert space localisation.  
\begin{defn} 
Let $M,N\subset X$  be complemented submodules. Let $\pi:B\to \mathbb{B}(H_{\pi})$ be a representation of $B$ on $H_{\pi}$. The quantity
\begin{equation}\label{localangle}
c_{\pi}(M,N):=c(M_{\pi},N_{\pi})=\|P_{M_{\pi}}P_{N_{\pi}}(1-P_{M_{\pi}\cap N_{\pi}})\|=\|P_{M_{\pi}}P_{N_{\pi}}-P_{M_{\pi}\cap N_{\pi}}\|,
\end{equation}
is called the (cosine of the) \emph{local Friedrichs angle between $M$ and $N$ at $\pi$}. 
\end{defn}

\begin{prop}
\label{independent} Suppose that $\pi:B\to \mathbb{B}(H_{\pi})$ is faithful. Then
\begin{enumerate}
\item If $(M,N)$ is concordant, then $c_{\pi}(M,N)=c(M,N)$;
\item If $(M,N)$ is discordant, then $c_{\pi}(M,N)=1$.
\end{enumerate}
In particular the (cosine of the) local Friedrich angle $c_{\pi}(M,N)$ is independent of the choice of faithful representation $\pi$.
\end{prop}
\begin{proof} Suppose $(M,N)$ is concordant, so that by Corollary \ref{locharmcor}, $M_{\pi}\cap N_{\pi}=(M\cap N)_{\pi}$ and $P_{M_{\pi}\cap N_{\pi}}=\widehat{\pi}(P_{M\cap N})$. 
 Since $\pi$ is faithful, the representation $\End^{*}_{B}(X)\to \mathbb{B}(X_{\pi})$ is faithful and hence isometric. Therefore 
\[c_{\pi}(M,N)=\|\widehat{\pi}(P_{N}P_{M}-P_{M\cap N})\|=\|P_{N}P_{M}-P_{M\cap N}\|=c(M,N),\]
which proves {\em1}.

Clearly $0\leq c_{\pi}(M,N)\leq 1$, so suppose that $c_{\pi}(M,N)<1$ and write $P=P_{M}$ and $Q=P_{N}$. We will show that the sequence $(PQ)^{n}$ is Cauchy for the norm topology. Then by Theorem \ref{complementedalternating}, $(M,N)$ is concordant, which proves {\em2}. So for $m\geq n$ recall the representation $\widehat{\pi}$ from Equation \eqref{eq:bob} and consider
\begin{align*}
\|(PQ)^{n}-(PQ)^{m}\|&=\|\widehat{\pi}((PQ)^{n}-(PQ)^{m})\|\\
&\leq \|\widehat{\pi}(PQ)^{n}-P_{M_{\pi}\cap N_{\pi}}\|+\|\widehat{\pi}(PQ)^{m}-P_{M_{\pi}\cap N_{\pi}}\|\\
&=\|(P_{M_\pi}P_{N_\pi})^{n}-P_{M_{\pi}\cap N_{\pi}}\|+\|(P_{M_\pi}P_{N_\pi})^{m}-P_{M_{\pi}\cap N_{\pi}}\|\\
&=\|(P_{M_\pi}P_{N_\pi}-P_{M_{\pi}\cap N_{\pi}})^{n}\|+\|(P_{M_\pi}P_{N_\pi}-P_{M_{\pi}\cap N_{\pi}})^{m}\|\quad(\textnormal{by Lemma \ref{powers}})\\
&\leq \|P_{M_\pi}P_{N_\pi}-P_{M_{\pi}\cap N_{\pi}}\|^{n}+\|P_{M_\pi}P_{N_\pi}-P_{M_{\pi}\cap N_{\pi}}\|^{m}\\
&=c_{\pi}(M,N)^{n}+c_{\pi}(M,N)^{m}\to 0,
\end{align*}
since $c_{\pi}(M,N)<1$. This completes the proof.
\end{proof}
We denote by $\widehat{B}$ the space of unitary equivalence classes of irreducible representations of the $C^{*}$-algebra $B$, by $\mathcal{P}(B)$ the pure state space of $B$ and by $\pi_{\sigma}$ the $GNS$-representation associated to the state $\sigma$. We can view the local Friedrichs angles as a function $\widehat{B}\to [0,1]$ and via the composition $\mathcal{P}(B)\to \widehat{B}$, also as a function on $\mathcal{P}(B)$.
\begin{corl} \label{localglobalangle} The Friedrichs angle \eqref{moduleangle} and the local Friedrichs angles \eqref{localangle} are related by $c(M,N)=\sup_{\pi\in\widehat{B}}c_{\pi}(M,N)=\sup_{\sigma\in\mathcal{P}(B)}c_{\pi_{\sigma}}(M,N)$.
\end{corl}
\begin{proof}
The representations $\widehat{H}=\bigoplus_{\pi\in\widehat{B}}H_{\pi}$ and $H_{\mathcal{P}}:=\bigoplus_{\sigma\in\mathcal{P}(B)}H_{\pi_{\sigma}}$ are faithful.
\end{proof}
In view of Proposition \ref{independent} and Corollary \ref{localglobalangle}, we \emph{define} the Friedrichs angle between an arbitrary pair of complemented submodules to be  $c(M,N):=c_{\pi}(M,N)$, with $\pi$ faithful. 
It was shown in \cite{Luo} that 
\begin{equation}
\label{Deutschangle}
c(M,N)=c(M^{\perp},N^{\perp}),
\end{equation}
provided that $M\cap N$ and $M^{\perp}\cap N^{\perp}$ are complemented. In particular, the equality holds for any pair of subspaces of a Hilbert space, \cite[Theorem 2.16]{D}. 
We will now show that the equality \eqref{Deutschangle} holds for an arbitrary pair of complemented submodules. This gives an extension, and a different proof, of \cite[Theorem 5.12]{Luo}.
\begin{thm}
\label{anglesymmetry}
Let $X$ be a Hilbert $C^{*}$-module and $M,N\subset X$ complemented submodules. Then $c(M,N)$ is well-defined and  $c(M,N)=c(M^{\perp},N^{\perp})$.
\end{thm}
\begin{proof}
For any representation $\pi:B\to\mathbb{B}(H_{\pi})$ there is an equality of submodules $(M_{\pi})^{\perp}=(M^{\perp})_{\pi}$ whenever $M$ is complemented. Moreover Equation \eqref{Deutschangle} holds for the subspaces $M_{\pi},N_{\pi}$ of the Hilbert space $X_{\pi}$. Thus by Proposition \ref{independent} we have
\begin{align*}
c(M,N)&=c_{\pi}(M,N)=c(M_{\pi},N_{\pi})\\ &=c((M_{\pi})^{\perp},(N_{\pi})^{\perp})=c((M^{\perp})_{\pi},(N^{\perp})_{\pi})\\&=c_{\pi}(M^{\perp},N^{\perp})=c(M^{\perp},N^{\perp}),\end{align*}
as claimed.
\end{proof}
Now we further analyse the properties of the local Friedrichs angles as a function on $\widehat{B}$.
\begin{prop}
\label{continuity}
Suppose $(M,N)$ is concordant. Then the map
\[\widehat{B} \to [0,1],\quad\pi\mapsto c_{\pi}(M,N),\]
is lower semi-continuous. If $X$ is full and $\widehat{B}$ is Hausdorff, $\pi\mapsto c_{\pi}(M,N)$ is continuous.
\end{prop}
\begin{proof} Let $J:=\left\langle B,B\right\rangle$ and $\widehat{B}\to\widehat{J}$ the restriction map, which is continuous. The $C^{*}$-algebras $J$ and $\mathbb{K}(X)$ are Morita equivalent, so by the Rieffel correspondence \cite{RieffelInd} the map $\pi\mapsto\widehat{\pi}$ is a homeomorphism $\widehat{J} \to \widehat{\mathbb{K}(X)}$. Since $\mathbb{K}(X)\subset \End^{*}_{B}(X)$ is an essential ideal, there is a continuous inclusion $\widehat{\mathbb{K}(X)}\to \widehat{\End^{*}_{B}(X)}$, see \cite[Section 2]{Dauns}. When $(M,N)$ is concordant the map $\pi\mapsto c_{\pi}(M,N)$ can be written as a composition
\[\pi\mapsto\widehat{\pi}\mapsto \|\widehat{\pi}(P_{M}P_{N}-P_{M\cap N})\|,\]
and is thus lower semicontinuous by \cite[Lemma A.30]{RW}. For $X$ full and $\widehat{B}$ Hausdorff, continuity follows by \cite[Lemma 5.2]{RW}.
\end{proof}
\begin{corl} 
Suppose $X$ is full, $B$ is unital, $\widehat{B}$ is Hausdorff and $(M,N)$ is concordant. Then $c(M,N)<1$ if and only if $c_{\pi}(M,N)<1$ for every irreducible representation $\pi$.
\end{corl}
\begin{proof} 
Since $\widehat{B}$ is compact Hausdorff and the Friedrichs angle is continuous, 
the pointwise estimate  $c_{\pi}(M,N)<1$ implies that $c(M,N)=\sup_{\pi\in\widehat{B}} c_{\pi}(M,N) <1$.
\end{proof}
\begin{rmk}
\label{discontangle}
In Proposition \ref{continuity}, the condition that $(M,N)$ be concordant cannot be relaxed. Consider $B=C([0,\frac{\pi}{2}])$, $X=C([0,\frac{\pi}{2}], \mathbb{C}^{2})$ and consider the submodules
\[
M=\Ran \begin{pmatrix} 1 & 0 \\ 0 & 0 \end{pmatrix},\quad N
= \Ran \begin{pmatrix} \cos^{2}t & \sin t \cos t \\ \sin t \cos t &\sin^{2}t\end{pmatrix}.
\]
For $t\in [0,\pi/2]$ we write $c_{t}(M,N)$ for the Friedrichs angle at $t$. For $0<t\leq \frac{\pi}{2}$ we have $M_{t}\cap N_{t}=\{0\}$ whereas at $t=0$ we have $M_0=N_0$. We thus have
\[
c_{t}(M,N)^{2}=\|P_{M_{t}}P_{N_{t}}\|^{2}=\left\|\begin{pmatrix}\cos^{2}t & \sin t \cos t \\ 0 & 0\end{pmatrix}\right\|^{2}
=\cos^{2}t <1,\quad 0<t\leq \frac{\pi}{2},
\]
and $c_{t}(M,N)=|\cos t|$, whereas $c_0(M,N)=0$. Thus the angle is discontinuous at $0$ and in particular we conclude once more that $(M,N)$ is discordant. This is another instance where the universal example (see Remark \ref{rmkuniversal}) provides a counterexample to a specific property. 

\end{rmk}

\subsection{Sum and intersection}
\label{sec: sumint}
With our extended definition of the Friedrichs angle we now examine the case $c(M,N)<1$ in more detail. Our results 
are all derived from results in \cite{Luo}, where complementability of $M\cap N$ is an assumption. We first recall the following well-known fundamental result and a relevant corollary.
\begin{thm}\label{range} Let $T\in \End^{*}_{B}(X)$, then
\[\overline{\textnormal{Ran}(T)}=\overline{\textnormal{Ran}(TT^{*})},\]
and $T$ has closed range if and only if $T^{*}$ has closed range. If $T$ has closed range then
\[\textnormal{Ran}(T)=\textnormal{Ran}(TT^{*}),\]
and $\textnormal{Ran}(T)$ is a complemented submodule of $X$ with $\textnormal{Ran}(T)^{\perp}=\ker T^{*}$.
\end{thm}
\begin{proof}
See \cite[Theorem 3.2, Proposition 3.7]{Lance}. 
\end{proof}
\begin{corl}[\cite{Luo}, Remark 5.8]\label{sumrange}
Let $P,Q$ be adjointable projections on a Hilbert $C^{*}$-module $X$, then 
\[\overline{\textnormal{Ran}(P)+\textnormal{Ran}(Q)}=\overline{\textnormal{Ran}(P+Q)}.\]
In particular
$\textnormal{Ran}(P)+\textnormal{Ran}(Q)$ is closed if and only if $\textnormal{Ran}(P+Q)$ is closed and in that case
\[\textnormal{Ran}(P)+\textnormal{Ran}(Q)=\textnormal{Ran}(P+Q),\]
which is a complemented submodule of $X$.
\end{corl}
\begin{proof}
By the following well-known observation
\[
\begin{pmatrix} 0 & 0 \\ 0 & P+Q\end{pmatrix}=\begin{pmatrix}0 & 0 \\ P & Q\end{pmatrix}\cdot \begin{pmatrix}0 & P \\ 0 & Q\end{pmatrix}=\begin{pmatrix}0 & 0 \\ P & Q\end{pmatrix}\cdot \begin{pmatrix}0 & 0 \\ P & Q\end{pmatrix}^{*},
\]
as operators on $X\oplus X$, it follows that $\overline{\textnormal{Ran}(P+Q)}=\overline{\textnormal{Ran}(P)+\textnormal{Ran}(Q)}$ by Theorem \ref{range}. Therefore,
 $\textnormal{Ran}(P)+\textnormal{Ran}(Q)$ is closed if and only if $\textnormal{Ran}(P)+\textnormal{Ran}(Q)=\textnormal{Ran}(P+Q)$ is closed.
\end{proof}
\begin{lemma}
\label{kernel}
Let $P,Q$ be adjointable projections on a Hilbert $C^{*}$-module $X$, then 
\[\ker(1-PQ)=\ker(1-QP)=\Ran P\cap \Ran Q.\]
\end{lemma}
\begin{proof} It is clear that $\Ran P\cap \Ran Q\subset \ker (1-PQ)\cap \ker (1-QP)$. We prove the reverse inclusion.
Suppose that $x\in \ker(1-PQ)$, so that $x=PQx$. Then clearly $x=Px$ and 
\[\langle Qx, Qx\rangle=\langle PQx, PQx\rangle+\langle (1-P)Qx,(1-P)Qx\rangle\geq \langle PQx,PQx\rangle=\langle x,x\rangle,\]
and thus $(1-Q)x=0$ so $x=Qx$. By reversing the roles of $P$ and $Q$ this shows that
\[\ker(1-PQ)=\ker(1-QP)=\{x\in X: x=Px=Qx\}=\Ran P\cap \Ran Q,\]
as claimed.\end{proof}
Lastly we recall the following remarkable result from \cite{LSX}, concerning the case where the internal sum $M+N$ is closed.
\begin{prop}[{\cite[Proposition 4.6]{LSX}}]\label{closedsum} Let $M$ and $N$ be complemented submodules of $X$, such that $M + N$ is closed in $X$. Then both $M + N$ and
$M \cap N$ are  complemented and $X=(M\cap N)\oplus (M^{\perp}+N^{\perp})$.
\end{prop}
In the above result, complementability of $M\cap N$ is concluded as opposed to assumed, and we now characterise closedness of $M+N$ in terms of our extended definition of $c(M,N)$, as well as in terms of properties of the operators $1-P_{M}P_{N}$ and $1-P_{N}P_{M}$.
\begin{prop}
\label{prop: closed range}
Let $X$ be a Hilbert $C^{*}$-module over a $C^{*}$-algebra $B$, and $M$ and $N$ complemented submodules. 
Then the following are equivalent: 
\begin{enumerate}
\item $c(M,N)<1$;
\item The sequence $(P_{M}P_{N})^{n}$ is Cauchy for the operator norm;
\item $M\cap N$ is complemented and $ \| P_{M}P_{N}-P_{M\cap N} \| = \|P_{N}P_{M}-P_{M\cap N}\|<1$;
\item $M\cap N$ is complemented and the operators $1-P_{M}P_{N}:(M\cap N)^{\perp}\to (M\cap N)^{\perp}$ and $1-P_{M}P_{N}:(M\cap N)^{\perp}\to (M\cap N)^{\perp}$ are bijective; 
\item $\textnormal{Ran}(1-P_{M}P_{N})$ and $\textnormal{Ran}(1-P_{N}P_{M})$ are closed;
\item $X=(M\cap N)\oplus (M^{\perp}+N^{\perp})$;
\item $M^{\perp}+N^{\perp}$ is closed;
\item $X=(M^{\perp}\cap N^{\perp})\oplus (M+N)$;
\item $M+N$ is closed;
\item $M\cap N$ is complemented and $(M\cap N)^{\perp}\cap M+(M\cap N)^{\perp}\cap N$ is closed.

\end{enumerate}
\end{prop}
\begin{proof}
We write $P=P_{M}$, $Q=P_{N}$ and $\Omega:=M\cap N$.

$1.\Rightarrow 2.$ This was shown in the proof of Proposition \ref{independent}.

$2.\Rightarrow 3.$ By Theorem \ref{complementedalternating}, $\Omega$ is complemented and $(PQ)^{n}\to P_{\Omega}$ in norm. By Lemma \ref{reductionlemma}, $(QPQ)^{n}\to P_{\Omega}$ in norm as well. Thus for $n$ sufficiently large $\|(QPQ)^{2^n}-P_{\Omega}\|<1$. Then 
applying Lemma \ref{powers} and the $C^{*}$-identity we find
\begin{align*}
\|(QPQ)^{2^n}-P_{\Omega}\|=\|(QPQ-P_{\Omega})^{2^{n}}\|=\|QPQ-P_{\Omega}\|^{2^{n}},
\end{align*}
and it follows that $\|QPQ-P_{\Omega}\|<1$. Now, again by the $C^{*}$-identity, \[\|PQ-P_{\Omega}\|^{2}=\|(QP-P_{\Omega})(PQ-P_{\Omega})\|=\|QPQ-P_{\Omega}\|,\] so we find that $\|PQ-P_{\Omega}\|<1$.

$3.\Rightarrow 4.$ Since $\| PQ-P_{\Omega}\|<1$. We find that the series
\[
\sum_{n=0}^{\infty}(PQ-P_{\Omega})^{n}=P_{\Omega}+\sum_{n=0}^{\infty}((PQ)^{n}-P_{\Omega})\quad\textnormal{(by Lemma \ref{powers})}, 
\]
is norm convergent to $(1+P_{\Omega}-PQ)^{-1}$. Since $\Omega=\ker (1-PQ)$ by Lemma \ref{kernel}, $1+P_{\Omega}-PQ$ is bijective and commutes with $P_{\Omega}$, it follows that it maps $\Omega^{\perp}$ bijectively onto itself. The same argument applies to $QP$.

$4.\Rightarrow 5.$ Since $X=\Omega\oplus \Omega^{\perp}$ and, by Lemma \ref{kernel}, $\Omega=\ker (1-PQ)=\ker (1-QP)$ it follows that 
\[\textnormal{Ran}(1-PQ)=\textnormal{Ran}(1-QP)=\Omega^{\perp},\] is closed.

$5.\Rightarrow 6.$
Since $\textnormal{Ran}(1-PQ)$ and $\textnormal{Ran}(1-QP)$ are closed, they are complemented by Theorem \ref{range}. Since $(1-PQ)^{*}=1-QP$, we have
\[\textnormal{Ran}(1-QP) =\ker(1-PQ)^{\perp}=\Omega^{\perp}=\ker(1-QP)^{\perp}=\textnormal{Ran}(1-PQ),\]
and by Lemma \ref{kernel} $\Omega=\ker(1-PQ)=\ker(1-QP)$.
Since
\[1-PQ=(1-P)Q+1-Q,\quad 1-QP=(1-Q)P+1-P,\]
it follows that 
\[
\Omega^{\perp}
=\textnormal{Ran}(1-PQ)\subset \textnormal{Ran}(1-P)+\textnormal{Ran}(1-Q)
\subset \Omega^{\perp}.
\]
Therefore
\[
\Omega^{\perp}=\textnormal{Ran}(1-P)+\textnormal{Ran}(1-Q)=M^{\perp}+N^{\perp}.
\]

$6.\Rightarrow 7.$ Since $\Omega^{\perp}$ is closed, $M^{\perp}+N^{\perp}$ is closed. 

$7.\Rightarrow 8.$ By Proposition \ref{closedsum} $M^{\perp}\cap N^{\perp}$ is complemented with complement $(M+N)^{\perp}$.

$8.\Rightarrow 9.$ Since $(M^{\perp}\cap N^{\perp})^{\perp}$ is closed $M+N$ is closed. 

$9.\Rightarrow 10.$ Since $\Omega$ is complemented, this follows from \cite[Lemma 5.11]{Luo}, and Corollary \ref{sumrange}  and Proposition \ref{closedsum} above. 

$10.\Rightarrow 1.$ By \cite[Lemma 5.11]{Luo} we have $c(M^{\perp}, N^{\perp})<1$, so by Theorem \ref{anglesymmetry} we have  $c(M,N)<1$.
\end{proof}
\begin{corl} Let $M,N\subset X$ be complemented submodules. Then $c(M,N)<1$  if and only if $M\cap N$ is complemented and $M+N$ is closed if and only if $M^{\perp}\cap N^{\perp}$ is complemented and $M^{\perp}+N^{\perp}$ is closed.
\end{corl}
\begin{rmk}
\label{Atiyah}
Let $E$ and $F$ be vector bundles over a compact Haursdorff space $Y$. In \cite[Definition 1.3.2]{Atiyah} a vector bundle morphism $\varphi:E\to F$ is called \emph{strict} if the map $y\mapsto \dim\ker\varphi_{y}$ is continuous, hence locally constant, on $Y$. By \cite[Proposition 1.32]{Atiyah}, if $\varphi$ is strict then $\bigcup_{y\in Y}\ker \varphi_{y}$ is a subbundle of $E$, and thus its module of sections is complemented.

As in Remarks \ref{eg:VB-unstrict} and \ref{discontangle} consider $B=C([0,\frac{\pi}{2}])$, $X=C([0,\frac{\pi}{2}], \mathbb{C}^{2})$ and consider the submodules
\[
M=\Ran \begin{pmatrix} 1 & 0 \\ 0 & 0 \end{pmatrix},\quad N
= \Ran \begin{pmatrix} \cos^{2}t & \sin t \cos t \\ \sin t \cos t &\sin^{2}t\end{pmatrix}.
\]
We obtain two globally trivial rank one vector bundles over $[0,\pi/2]$ whose modules of sections are $M$ and $N$ respectively. 
The vector bundle homomorphism $\varphi:X\to X$ defined by $\varphi:=(1-P_{M}P_{N})$
is not strict since
\[
\ker\varphi_t=\ker (1-P_{M_t}P_{N_t})=M_{t}\cap N_{t},
\]
so $t\mapsto\dim\ker\varphi_{t}$ is discontinuous at $0$ by Remark \ref{eg:VB-unstrict}. 
\end{rmk}
For a commutative unital $C^{*}$-algebra $B$ the following corollary to Proposition \ref{prop: closed range} corresponds to the situation where cosine of the angle between the corresponding sub-vector bundles is $<1$ in which case the bundle endomorphism $1-P_{M}P_{N}$ is strict.
\begin{corl}If the $C^{*}$-algebra $B$ is unital and $X$ is finitely generated and projective over $B$, then for any pair $M,N\subset X$ of complemented submodules with $c(M,N)<1$ both $M\cap N$ and $M+N$ are finitely generated projective $B$-modules.
\end{corl}
\begin{proof}
If $X$ is finitely generated and projective over the unital algebra $B$, then any complemented submodule is finitely generated and projective, so in particular $M\cap N$ is. Since $M+N$ is finitely generated and closed in $X$, it is a Hilbert $C^{*}$-module and hence is projective by \cite[Theorem 1.4.6]{Troitsky}.
\end{proof}

\end{document}